\documentclass[12pt]{iopart}
\usepackage{iopams}  

\usepackage{graphicx}

\expandafter\let\csname equation*\endcsname\relax

\expandafter\let\csname endequation*\endcsname\relax

\usepackage{amsmath}
\usepackage{amstext,amsthm}

\newcommand{\be}{\begin{equation}}
\newcommand{\ee}{\end{equation}}

\def\R{{\mathbb R}}
\def\C{{\mathbb C}}
\def\N{{\mathbb N}}

\def\E{{\mathbb E}}

\newcommand{\gl}{\lambda}
\newcommand{\glN}{\lambda^{(N)}}

\newcommand{\PP}{{\mathcal{P}}}
\newcommand{\Beta}{{\mathrm{B}}}
\newcommand{\Unif}{{\mathrm{U}}}

\newtheorem{thm}{Theorem}

\newtheorem{cor}{Corollary}
\newtheorem{prop}{Proposition}
\newtheorem{dfn}{Definition}
\newtheorem{claim}{Claim}
\newtheorem{lem}{Lemma}
\newtheorem{ex}{Example}
\theoremstyle{definition}

\theoremstyle{remark}
\newtheorem{rem}{Remark}

\newcommand{\staircase}{
\raisebox{-1pt}{\setlength{\unitlength}{0.12in}
\!\begin{picture}(1,1)(0,0)
\put(0,0){\line(1,0){0.33}}
\put(0.33,0){\line(0,1){1}}
\put(0.,0.33){\line(1,0){0.66}}
\put(0.66,0.33){\line(0,1){0.66}}
\put(0,0.66){\line(1,0){1}}
\put(1,0.66){\line(0,1){0.33}}
\put(0,0){\line(0,1){1}}
\put(0,1){\line(1,0){1}}
\end{picture}} 
}

\newcommand{\boxp}{
\raisebox{-1pt}{\setlength{\unitlength}{0.12in}
\!\begin{picture}(1,1)(0,0)
\put(0,0){\line(1,0){1}}
\put(0,0){\line(0,1){1}}
\put(0.33,0){\line(0,1){1}}
\put(0.66,0){\line(0,1){1}}
\put(0,1){\line(1,0){1}}
\put(0,0.33){\line(1,0){1}}
\put(0,0.66){\line(1,0){1}}
\put(1,0){\line(0,0){1}}
\end{picture}}\, 
}

\begin{document}

\title[Random matrices associated to Young diagrams]{Random matrices associated to Young diagrams}

\author{Fabio Deelan Cunden$^{1,2}$, Marilena Ligab\`o$^{1}$,  Tommaso Monni$^{1}$}
\address{1. Dipartimento di Matematica, 
 Universit\`a degli Studi di Bari, 
 I-70125 Bari, 
  Italy\\
  2. INFN, Sezione di Bari, I-70125 Bari, Italy}
\ead{fabio.cunden@uniba.it}
\ead{marilena.ligabo@uniba.it}
\ead{tommaso.monni@uniba.it}
\vspace{10pt}

\begin{abstract}
We consider the singular values of certain Young diagram shaped random matrices. 
For block-shaped random matrices,  the empirical distribution of the squares of the singular eigenvalues converges almost surely to a distribution whose moments are a generalisation of the Catalan numbers. The limiting distribution is the density of a product of rescaled independent Beta random variables and its Stieltjes-Cauchy transform has a hypergeometric representation. In special cases we recover the Marchenko-Pastur and Dykema-Haagerup measures of square and triangular random matrices, respectively. We find a further factorisation of the moments in terms of two complex-valued random variables that generalises the factorisation of the Marcenko-Pastur law as product of independent uniform and arcsine random variables. 
\end{abstract}

%
%
%
%
%

\section{Introduction}

Let $X$ be a matrix whose entries are i.i.d.\ complex random variables. The nonnegative definite matrix   $X X^*$
is known as \emph{random covariance matrix}, and it is arguably one of the most studied models in random matrix theory with varied applications in physics, statistics and other areas~\cite{Forrester10,Livan18,Muirhead09}. 
\par

In this paper we consider a class of random matrices, that we dub \emph{$\lambda$-shaped random matrices}.  They can be thought of as a generalisation of random covariance matrices where $X$ has the `shape' of a Young diagram. Instances of these matrix models in the special case of Gaussian entries were studied by Dykema and Haagerup~\cite{Dykema04} as a tool to construct certain non commutative random variables called \emph{DT-elements}, and in the work of F\'eray and Sniady~\cite{Feray11} in relation to \emph{Stanley's character formula}~\cite{Stanley04} of the symmetric group. Under various names and in disguised form, Gaussian $\lambda$-shaped matrices recently resurfaced in connection to biorthogonal ensembles, last passage percolation and free probability~\cite{Adler13,Cheliotis18,Collins14,Forrester17,Nakashima20}.
\par

The main result of this paper (Theorem~\ref{thm:main}) is a limit theorem for the empirical distribution of the eigenvalues of block-shaped random matrices. We also obtain various properties of the limiting distribution (Proposition~\ref{prop:G_Beta} and Corollary~\ref{cor:supp}).

\section{Partitions and Young diagrams}

We begin by reviewing some of the basic terminology of integer partitons and Young diagrams. A standard reference is~\cite[Ch. I]{Macdonald}.

\par

A \emph{partition} is any (finite or infinite) sequence $\gl=(\lambda_1,\lambda_2,\ldots,\lambda_m,\ldots)$ of weakly decreasing nonnegative integers, $\lambda_1\geq\lambda_2\geq\cdots\geq\lambda_m\geq\cdots$. It is convenient not to distinguish between two such sequences which differ only by a string of zeros at the end. Thus, for example, we regard $(5,4,4,1)$, $(5,4,4,1,0)$, $(5,4,4,1,0,0,\ldots)$ as the same partition.  We denote the set of partitions by $\mathcal{P}$. The number of nonzero elements $\lambda_i$ of a partition $\lambda$ is called \emph{number of parts} or \emph{length} of $\lambda$, denoted by $\ell(\lambda)$; the sum of the parts is the \emph{weight} of $\lambda$, denoted by $|\lambda|=\lambda_1+\lambda_2+\cdots$. If $|\gl|=n$, we say that $\gl$ is a partition of $n$ and we write $\lambda\vdash n$. 
\par

The \emph{Young diagram} of a partition $\lambda$ is the set of points $\{(i,j)\in \N^2\colon1\leq j\leq \lambda_i\}$. A Young diagram is drawn usually as a set of boxes, not of points.  In drawing such diagrams we shall adopt the \emph{English convention}, as with matrices, that the first coordinate $i$ (the row index)  increases as one goes downwards, and the second coordinate $j$ (the column index) increases from left to right. With this convention the diagram can be visualised as a diagram of {left-justified} rows of boxes where the $i$-th row contains $\lambda_i$ boxes (hence each row is {not longer} than the row on top of it).  For instance, the diagram of $\gl = (5,4,4,1)\vdash 14$ is:
\begin{center}\setlength{\unitlength}{0.011in}\linethickness{0.0125in}
\begin{picture}(100,80)
\multiput(0,80)(0,-20){2}{\line(1,0){100}} %
\multiput(0,40)(0,-20){2}{\line(1,0){80}}%
\multiput(0,0)(0,-20){1}{\line(1,0){20}}%
\multiput(0,0)(20,0){2}{\line(0,1){80}}%
\multiput(40,20)(20,0){3}{\line(0,1){60}}%
\multiput(100,60)(20,0){1}{\line(0,1){20}}%
\end{picture}
\end{center}
It is customary to identify a partition $\lambda\vdash n$ and its Young diagram with $n$ boxes.
\par

For $\lambda,\mu\in\PP$, we shall write $\mu\subset \gl$ to mean that the diagram of $\gl$ contains the diagram of $\mu$, i.e. that $\mu_i\leq \gl_i$ for all $i\geq1$. We write $(i,j)\in\gl$ to mean that the box $(i,j)$ is in the diagram of $\gl$, i.e. $\gl_i\geq j$. The conjugate of a partition $\gl$ is the partition $\gl'$ whose diagram is the  transpose of the diagram $\gl$, i.e. the diagram obtained by reflection in the main diagonal. 
\par

It is possible to define an operation of left multiplication by positive integers on the set of Young diagrams. (In fact, it is possible to define a multiplication by positive real numbers using a natural identification between diagrams and their border path when drawing a partition in `Russian notation'~\cite{Dolega}.) 
If $N\in\mathbb{N}$ and $\lambda\in\PP$, then $N\lambda$ is the partition with $N\ell(\gl)$ parts given by
$$
(N\lambda)_i=N\lambda_j \quad \mathrm{iff}\quad jN \leq i<(j+1)N .
$$
In other words
$$
N\lambda=\left(\underbrace{N\lambda_1,\ldots,N\lambda_1}_{N\,\mathrm{times}},\underbrace{N\lambda_2,\ldots,N\lambda_2}_{N\,\mathrm{times}},\ldots\right).
$$
The Young diagram of $N\lambda$ is a \emph{dilation} of the Young diagram of $\lambda$ obtained  by replacing each box in $\lambda$ by a grid of $N\times N$ boxes. Hence, if $\lambda\vdash n$, then $N\lambda\vdash N^2 n$.   
For instance, if $\gl = (5,4,4,1)\vdash 14$, then  $3\lambda=(15,15,15,12,12,12,12,12,12,3,3,3)\vdash 3^2\cdot 14$. Its diagram is:
\begin{center}
\begin{picture}(300,200)\setlength{\unitlength}{0.01in}
\multiput(200,240)(0,-20){4}{\line(1,0){300}} %
\multiput(200,160)(0,-20){6}{\line(1,0){240}}%
\multiput(200,0)(0,20){3}{\line(1,0){60}}%
\multiput(200,0)(20,0){4}{\line(0,1){240}}%
\multiput(280,60)(20,0){9}{\line(0,1){180}}%
\multiput(460,180)(20,0){3}{\line(0,1){60}}%
\linethickness{0.0125in}
\multiput(0,180)(0,-20){2}{\line(1,0){100}} %
\put(-40,140){ \large{$3\,\,\cdot$}}%
\multiput(0,140)(0,-20){2}{\line(1,0){80}}%
\multiput(0,100)(0,-20){1}{\line(1,0){20}}%
\multiput(0,100)(20,0){2}{\line(0,1){80}}%
\multiput(40,120)(20,0){3}{\line(0,1){60}}%
\multiput(100,160)(20,0){1}{\line(0,1){20}}%
\put(140,130){ \large{$=$}}%
\multiput(200,240)(0,-60){2}{\line(1,0){300}} %
\multiput(200,60)(0,60){4}{\line(1,0){240}} %
\multiput(200,00)(0,0){1}{\line(1,0){60}} %
\multiput(200,00)(60,0){2}{\line(0,1){240}} %
\multiput(260,60)(60,0){4}{\line(0,1){180}} %
\multiput(500,180)(60,0){1}{\line(0,1){60}} %
\end{picture}
\end{center}

\par

We use the following pictorial notation. The square partition with $r$ parts all equal to $r$ is denoted $\boxp_r$. For the staircase partition $(r,r-1,\dots,2,1)$  we use the symbol $\staircase_r$.

 \section{$\lambda$-shaped random matrices}
Fix a field $K$, and a partition $\gl$. We denote by $M(K,\gl)$ the set of all $\ell(\gl)\times\ell(\gl')$ matrices $A=(a_{ij})$ over $K$ with entries $a_{ij}=0$ if $(i,j)\notin\lambda$. The set $M(K,\gl)$ is a $K$-vector space. These sets appeared in applied mathematics and computer science mostly when $K$ is a finite field. See~\cite{Ballico15} and references therein. We call matrices elements of $M(K,\gl)$,  \emph{$\gl$-shaped matrices}.
\par

In the present paper we will consider spectral properties of random matrices in $M(\C,\gl)$ defined as follows. 
Let $\left(\glN\right)_{N\geq1}$ a nested sequence in $\mathcal{P}$,
\be
\gl^{(1)}\subset \gl^{(2)} \subset \cdots\subset \lambda^{(N)}\subset\cdots.
\ee
\par

 Let $\{X_{ij}\colon i,j\in\mathbb{N}\}$ be i.i.d.\ complex random variables with zero mean $\E X_{ij}=0$ and second moment $\E |X_{ij}|^2=1$. For $N\in\N$ let  $X_N$ be the  $\glN$-shaped matrix whose $(i,j)$-th entry is  $X_{ij}$ if $(i,j)\in \gl^{(N)}$ and $0$ otherwise. 
\par
Set $\ell_N:=\ell(\glN)$ and  consider the $\ell_N\times \ell_N$ complex Hermitian matrix
\be
\label{eq:defW}
W_N:=\frac{1}{N} X_N X_N^*.
\ee
Denote the eigenvalues of $W_N$ by $x_1^{(N)}\leq x_2^{(N)}\leq\cdots\leq x_{\ell_N}^{(N)}$, 
and 
\be
\label{eq:defF}
F_N(x):=\frac{1}{\ell_N}\#\{j\leq \ell_N\colon x_j^{(N)}\leq x \}
\ee
the empirical distribution of the eigenvalues $x_i^{(N)}$ of $W_N$. It is natural to ask whether the sequence $(F_N)_{N\geq1}$ converges in distribution. The limit, when it exists, will be referred to as the \emph{limiting spectral distribution}  of $W_N$.

We begin by recalling two known cases. 

\subsection{Full matrices}
Classical Wishart matrices fit naturally in the setting of $\lambda$-shaped random matrices. Let $\gl^{(N)}=\boxp_N$. We have 
$$
\boxp_1\subset \boxp_2\subset\cdots\subset \boxp_N\subset\cdots
$$

Then, the $\boxp_N$-shaped random matrix $X_N$ is simply a $N\times N$ matrix with i.i.d. entries $X_{ij}$, so that $W_N$ is a Wishart matrix. Recall that the \emph{Catalan numbers} are 
\be
C_k = \frac{1}{k+1}{2k \choose k}.
\label{eq:Catalan}
\ee
Catalan numbers count hundreds of different combinatorial objects~\cite{Stanley15}, such as \emph{rooted plane trees} of $k+1$ vertices.
It is a classical result that $(F_N)_{N\ge1}$ converges in distribution to a deterministic distribution function  $F_{\rm MP}$ whose moments are the Catalan numbers
\begin{equation}
\int_\R x^k dF_{\rm MP}(x)=C_k.
\end{equation}
The limiting spectral distribution $F_{\rm MP}$ is called \emph{Marchenko-Pastur distribution}~\cite{Marchenko67}. It is supported on the interval $[0,4]$ with density
\be
F'_{\rm MP}(x)= \frac{1}{2\pi}\sqrt{\frac{4-x}{x}}\chi_{[0,4]}(x).
\label{eq:MP}
\ee

\subsection{Triangular matrices}
Let $\glN=\staircase_N$ be the staircase partition of length $N$.  Again $\staircase_N\subset \staircase_{N+1}$ for all $N$. 
Then, $X_N$ is a $N\times N$ \emph{triangular} random matrix (with 
entry $(i,j)=X_{ij}$ if $i+j\leq N$ and $0$ otherwise). These matrices where first considered by Dykema and Haagerup~\cite{Dykema04} who proved the existence of the limiting spectral distribution  $F_{\rm DH}$ 
whose moments are
\begin{equation}
\int_\R x^k dF_{\rm DH}(x)=\frac{1}{k+1}\frac{k^k}{k!}.
\label{eq:momentsDH}
\end{equation}
The \emph{Dykema-Haagerup distribution} $F_{\rm DH}$ comes from a density supported on the interval $[0,\rme]$ and defined by (see~\cite[Theorem 8.9]{Dykema04})
\be
F'_{\rm DH}\left(\frac{\sin v}{v} \exp(v \cot v)\right) = \frac{1}{\pi} \sin v\exp(-v\cot v), \quad 0<v<\pi.
\ee
This density can be also written in terms of Lambert function, see~\cite[Corollary 1]{Cheliotis18} (arXiv version)  and~\cite[Remark 3.9.]{Forrester17}.
\subsection{Balanced shapes}
We would like to study spectral properties of $\gl$-shaped random matrices for more general increasing sequence $\left(\gl^{(N)}\right)_{N\geq1}$ of Young diagrams. This amounts to understand the large-$N$ limit of \emph{moments}
\be
\lim_{N\to\infty}\frac{1}{\ell_N}\E\Tr W_N^k,\quad k=0,1,2,\dots,
\ee
where $W_N$ is defined is~\eref{eq:defW}.
We expect to find a nontrivial  limiting spectral distribution if the sequence $\glN$ `converges' to a limit shape (the `macroscopic shape' of $\glN$). The first moment ($k=1$) calculation can makes this a bit more precise,
\be
\frac{1}{\ell_N}\E\Tr W_N=\frac{1}{\ell_N }\frac{1}{N}\sum_{(i,j)\in\glN}\E|X_{ij}|^2
=\frac{\left|\glN\right|}{N \ell(\glN)}.
\ee 
Therefore, in order to have a nontrivial limit distribution one needs (at least) the length of the partitions to scale like the square root of the weight 
\be
N\ell_N \sim |\glN|,\quad \mathrm{as}\quad N\to\infty.
\label{eq:growth}
\ee
Young diagrams satisfying such a  \emph{growth condition} are  called \emph{balanced Young diagrams}~\cite{Dolega}. If we view a Young diagram as a geometric object, Eq.~\eref{eq:growth}  suggests to consider sequences that, after rescaling as $\frac{1}{N\ell_N }\glN$, tend to a limit shape $\lambda$.
\par
The easiest example of balanced Young diagrams is the sequence of  dilations of a fixed partition. Let $\gl\in\PP$ and consider the sequence $\glN=N\gl$. For such a sequence the ratio $\frac{\left|\glN\right|}{N \ell_N}=\frac{\left|\gl\right|}{\ell(\gl)}$ is constant. 
\par
\section{Block-shaped random matrices}
Fix a positive integer $r$, consider the  staircase partition $\staircase_r$ of length $r$. Define the sequence $\glN=N\staircase_r$. In this case $\ell_N=Nr$, and $\glN\subset\gl^{(N+1)}$ for all $N$. The matrix $X_N$ is a \emph{block-shaped random matrix} with ${r+1 \choose 2}$ nonzero blocks of size $N\times N$.
\begin{ex} For $r=3$, we have
\small{
\begin{gather*}
X_1=
\begin{pmatrix}
X_{11} &X_{12}  &X_{13}\\
X_{21} & X_{22} &0\\
X_{31} & 0&0
\end{pmatrix}
,\quad
X_2=
\begin{pmatrix}
X_{11} &X_{12}  &X_{13} &X_{14} &X_{15}  &X_{16}\\
X_{21} &X_{22}  &X_{23} &X_{24} &X_{25}  &X_{26}\\
X_{31} &X_{32}  &X_{33} &X_{34} &0 &0\\
X_{41} &X_{42}  &X_{43} &X_{44} &0 &0\\
X_{51} & X_{52} &0 &0 &0 &0\\
X_{61} & X_{62} &0 &0 &0 &0\\
\end{pmatrix},\\
X_3=
\begin{pmatrix}
X_{11} &X_{12}  &X_{13} &X_{14} &X_{15}  &X_{16}&X_{17} &X_{18}  &X_{19}\\
X_{21} &X_{22}  &X_{23} &X_{24} &X_{25}  &X_{26}&X_{27} &X_{28}  &X_{29}\\
X_{31} &X_{32}  &X_{33} &X_{34} &X_{35}  &X_{36}&X_{37} &X_{38}  &X_{39}\\
X_{41} &X_{42}  &X_{43} &X_{44} &X_{45}  &X_{46}&0&0  &0\\
X_{51} &X_{52}  &X_{53} &X_{54} &X_{55}  &X_{56}&0&0  &0\\
X_{61} &X_{62}  &X_{63} &X_{64} &X_{65}  &X_{66}&0 &0  &0\\
X_{71} &X_{72}  &X_{73} &0 &0  &0&0&0  &0\\
X_{81} &X_{82}  &X_{83} &0 &0  &0&0 &0  &0\\
X_{91} &X_{92}  &X_{93} &0 &0 &0&0 &0  &0
\end{pmatrix}, \quad \text{etc.}
\end{gather*}
}
\end{ex}
 Let $F_N$ be the empirical distribution~\eref{eq:defF} of the eigenvalues of $W_N$. 
\begin{thm} 
\label{thm:main} Let $\glN=N\staircase_r$. Then, the sequence $\left(F_N\right)_{N\geq1}$ converges, with probability $1$, to the deterministic distribution $F_{\langle r\rangle}$ with moments
\be
\int_\R x^k dF_{\langle r\rangle}(x)=\frac{1}{k+1}{(r+1)k \choose k}.
\ee
\end{thm}
\par
 
The limiting moments $m_k$ (multiplied by $r$) have a combinatorial interpretation. They enumerate plane trees whose vertices are given labels from the set $\{1, \ldots, r\}$ in
such a way that the sum of the labels along any edge is at most $r+1$. These combinatorial object were invented by Gu, Prodinger and Wagner~\cite{Gu10}, extending a previous definition by Gu and Prodinger~\cite{Gu09}. 
\begin{dfn}
A $r$-plane tree is a pair $(T,c)$, where $T=(V,E)$ is a plane tree, and $c\colon V\to \{1, \ldots, r\}$ is a colouring such that $c(u)+c(v)\leq r+1$ whenever $\{u,v\}\in E$. \end{dfn}
\begin{prop}[Gu, Prodinger and Wagner~\cite{Gu10}]
\label{prop:Gu}
 The number of $r$-plane trees on $k+1$ vertices is
\be
C^{\langle r\rangle}_k=\frac{r}{k+1}{(r+1)k \choose k}.
\ee
\end{prop}
For recent refined formulae see~\cite{Okoth22}. 
The integers $\left(C^{\langle r\rangle}_k\right)_{k\geq0}$ are \emph{generalised Catalan numbers}. 
Here are a few values of them.
$$
\begin{array}{c|rrrrrr}
k & C^{\langle 1\rangle}_k &C^{\langle 2\rangle}_k &C^{\langle 3\rangle}_k &C^{\langle 4\rangle}_k &C^{\langle 5\rangle}_k &C^{\langle 6\rangle}_k\\
0& 1 & 2 & 3 & 4 & 5 & 6 \\
1& 1 & 3 & 6 & 10 & 15 & 21 \\
2& 2 & 10 & 28 & 60 & 110 & 182 \\
3& 5 & 42 & 165 & 455 & 1020 & 1995 \\
4& 14 & 198 & 1092 & 3876 & 10626 & 24570 \\
5& 42 & 1001 & 7752 & 35420 & 118755 & 324632 \\
6& 132 & 5304 & 57684 & 339300 & 1391280 & 4496388 \\
7& 429 & 29070 & 444015 & 3362260 & 16861455 & 64425438 \\
8& 1430 & 163438 & 3506100 & 34179860 & 209638330 & 946996050 \\
9& 4862 & 937365 & 28242984 & 354465254 & 2658968130 & 14200613889 \\
10& 16796 & 5462730 & 231180144 & 3735373880 & 34270012530 & 216384285936 \\
\end{array}
$$
The sequences $\left(C^{\langle 1\rangle}_k\right)_{k\geq0}$ for $r=1,2,3,4$ are the entries A000108, A007226, A007228, and A124724 in The On-Line Encyclopedia of Integer Sequences~\cite{OEIS}.
\begin{rem} For $r=1$, the moments coincide with the Catalan sequence
\be
C^{\langle 1\rangle}_k=\frac{1}{k+1}{2k \choose k}.
\ee
\par

For $r=2$,  $X_N$ is a `three-blocks' matrix model studied by Flynn-Connolly~\cite{Flynn-Connolly18} who proved that the moments are $\frac{1}{k+1}{3k\choose k}$. The sequence entry A005132 in The On-Line Encyclopedia of Integer Sequences.

\par
For large $r$, we recover the moments~\eqref{eq:momentsDH} of the Dykema-Haagerup measure,
\be
\lim_{r\to\infty}\frac{1}{r^{k+1}}C^{\langle r\rangle}_k=\frac{1}{k+1}\frac{k^k}{k!}.
\ee
\end{rem}

We now present a few results on the limiting measure. They are motivated by the observation (by Ledoux~\cite{Ledoux04}) that a semicircular variable is equal in distribution to the product of the square root of a uniform random variable and an independent arcsine random variable. We denote by $\Unif(0,\ell)$ a random variable uniformly distributed in the interval $[0,\ell]$, and simply by $\Unif$ a random variable uniformly distribute in the unit interval $[0,1]$. With $\Beta(a,b)$ we generally denote a beta random variable with parameters $a,b>0$; it has support in the interval $[0,1]$, with density 
\be
\displaystyle\frac{\Gamma(a+b)}{\Gamma(a)\Gamma(b)}x^{a-1}(1-x)^{b-1},
\ee
where $\Gamma(\cdot)$ is the Euler Gamma function. Let $Y_{\rm MP}$ be a random variable with Marchenko-Pastur distribution~\eqref{eq:MP}. Then, $Y_{\rm MP}$ is equal in distribution to the product of a uniform random variable on the interval $[0,4]$ and an independent arcsine variable in the interval $[0,1]$. In formulae:
\be
Y_{\rm MP}\overset{\rm d}{=} \Unif(0,4)\Beta(1/2,1/2).
\label{eq:factorization_MP}
\ee
This is equivalent to the above mentioned factorisation of semicircular variables~\cite{Ledoux04}. Indeed, if $\Unif,\Unif'$ are independent and uniformly distributed on $[0,1]$, we can write 
\be
\E Y_{\rm MP}^k=  \E(\Unif A)^k,
\ee 
where $A\overset{\rm d}{=} \left(2\cos\pi \Unif'\right)^{2}$. The random variable $A$ is the rescaled squared projection of a uniform point on the unit semicircle, hence an arcsine random variable.
See~\cite{Cunden19} for a `semiclassical' interpretation for Gaussian random matrices.
\begin{prop} 
\label{prop:G_Beta}
Let $r\geq 1$. Set $L(r)= \dfrac{(r+1)^{r+1}}{r^r}>0$.
\begin{enumerate}
\item The Stieltjes-Cauchy transform of $F_{\langle r\rangle}$
\be
G_{\langle r\rangle}(z):=\int_{\R}\frac{1}{z-x}\rmd F_{\langle r\rangle}(x),
\ee
has the hypergeometric representation
\be
\label{eq:Hypergeometric}
G_{\langle r\rangle}(z)=
\frac{1}{r+1}\left(1-{}_{r}F_{r-1} \left[ \begin{matrix}
 -\frac{1}{r+1},-\frac{2}{r+1}, \ldots , -\frac{r}{r+1} \\ 
-\frac{1}{r} , -\frac{2}{r}, \ldots , -\frac{r-1}{r} \end{matrix} ;\frac{L(r)}{z} \right]\right)
\ee
\item 
If $Y_{\langle r \rangle}$ is a real random variable with distribution $F_{\langle r\rangle}$, then the following identity in distribution holds
\begin{align}
\label{eq:prod_Beta}
Y_{\langle r \rangle}&\overset{\rm d}{=}\Unif\left(0,L(r)\right)\prod_{j=1}^r\Beta\left(\frac{j}{r+1},\frac{j}{r(r+1)}\right),
\end{align}
where the variables on the right are jointly independent.  
\item Let $\Unif,\Unif'$ be independent and uniformly distributed on $[0,1]$. Then,
\be
\E Y_{\langle r \rangle}^k=\E\left(\Unif A_{\langle r\rangle}\right)^k.
\label{eq:moments_proj}
\ee
where $A_{\langle r\rangle}\overset{\rm d}{=} e^{(r-1)\pi \rmi \Unif'}\left(2\cos\pi  \Unif'\right)^{r+1}$.
\end{enumerate}
\end{prop}

\begin{rem}
For the first values of $r\geq1$ we get
\begin{align}
Y_{\langle 1\rangle}&\overset{\rm d}{=}\Unif\left(0,\frac{2^2}{1^1}\right)\Beta\left(\frac{1}{2},\frac{1}{2}\right),\label{eq:Y1}\\
Y_{\langle2\rangle}&\overset{\rm d}{=}\Unif\left(0,\frac{3^3}{2^2}\right)\Beta\left(\frac{1}{3},\frac{1}{2\cdot3}\right)\Beta\left(\frac{2}{3},\frac{2}{2\cdot3}\right),\label{eq:Y2} \\
Y_{\langle3\rangle}&\overset{\rm d}{=}\Unif\left(0,\frac{4^4}{3^3}\right)\Beta\left(\frac{1}{4},\frac{1}{3\cdot4}\right)\Beta\left(\frac{2}{4},\frac{2}{3\cdot4}\right)\Beta\left(\frac{3}{4},\frac{3}{3\cdot4}\right),
\end{align}
etc.
For $r=1$, Eq.~\eref{eq:Y1} coincides with the factorisation~\eqref{eq:factorization_MP}. 
For $r=2$, Eq.~\eqref{eq:Y2} is equivalent to a formula proved by M\l otkowski and Penson~\cite[Proposition 3.1]{Mlotkowski13}.

We can use the decomposition~\eref{eq:prod_Beta} to write `explicit' expressions for the densities $F'_r$. For small values of $r$ we have
\begin{align}
F'_{\langle 1\rangle}(x)&=\frac{1}{2\pi}\sqrt{\frac{L(1)-x}{x}}\chi_{[0,L(1)]}(x)\\
F'_{\langle 2\rangle}(x)&=\frac{1}{2^{3+1/3} 3^{1/2} \pi  w^{2/3}}\Bigl[\left(3^{1/2}+ 2 \sqrt{L(2)-w}\right)\left(3^{3/2}-2 \sqrt{L(2)-w}\right)^{1/3}\nonumber\\
&-\left(3^{1/2}-2 \sqrt{L(2)-w}\right) \left(3^{3/2}+2 \sqrt{L(2)-w}\right)^{1/3}\Bigr]\chi_{[0,L(2)]}(w).
  \label{eq:density_W'}
\end{align}
The first is the Marchenko-Pastur distribution $F'_{\rm MP}(x)$. 
The second is equivalent to formula in~\cite[Theorem 3.1]{Mlotkowski13} proved by Mellin inversion. 
For generic values of $r$,  a `direct' way to numerically compute $F'_{\langle r\rangle}(x)$ is by applying the Stieltjes inversion formula
\be
F_{\langle r\rangle}'(x)=-\frac{1}{\pi}\lim_{\epsilon\searrow0} \operatorname{Im}G_{\langle r\rangle}(x+\rmi \epsilon),
\ee
to Equation~\eref{eq:Hypergeometric}. This is how we made the numerical plots shown in Fig.~\ref{fig:densities}. 
\end{rem}
In fact, one could, as in~\cite{Penson11}, write $F_{\langle r\rangle}'(x)$ as the inverse Mellin transform of the moments. Since the moments are products of ratios of Pochhammer symbols, the resulting density would be a Mejier-G function (see the note by Dunkl~\cite{Dunkl13}). In fact, from the previous proposition we can get some precise information on the support of $F_{\langle r\rangle}'(x)$ and on its behavior at the edges.

 \begin{figure}[t]
	\centering
	\includegraphics[width=1\textwidth]{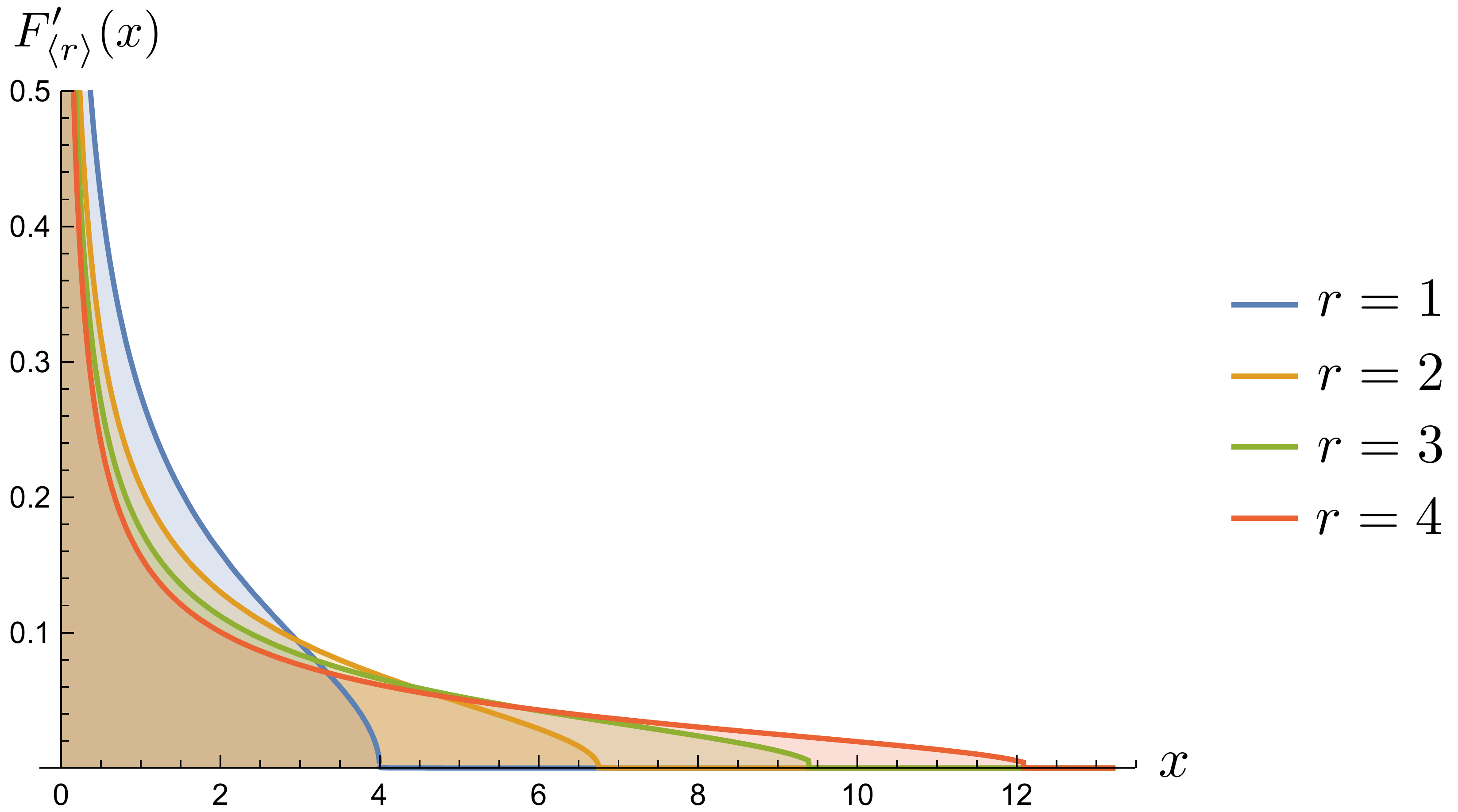}
\caption{Plot of the densities $F'_{\langle r\rangle}(x)$ for several values of $r$.}
\label{fig:densities}
\end{figure}

\begin{cor}
\label{cor:supp}
The measure $F_{\langle r\rangle}$ has a density, $F'_{\langle r\rangle}$ with support
\be
\mathrm{supp}F'_{\langle r\rangle}:=\overline{\left\{x\colon F'_{\langle r\rangle}(x)>0\right\}}=\left[0,L(r)\right].
\ee
Moreover at the edges of the support the density behaves as
\be
F'_{\langle r\rangle}(x)\sim 
\begin{cases}
\displaystyle c\, x^{-\frac{r}{r+1}},&\text{as}\quad x\searrow 0,\\\\
 \displaystyle c'\,\sqrt{L(r)-x},&\text{as}\quad x\nearrow L(r),
\end{cases}
\ee
for some constants $c, c'>0$.
\end{cor}
The density vanishes as a square root at the  `soft edge' $x=L(r)$. At the `hard edge' $x=0$ the density has an integrable singularity $x^{-\frac{r}{r+1}}$. For $r=1$ (Wishart matrices) this is the classical $x^{-\frac{1}{2}}$ divergence, but for $r>1$ the singularity is stronger. 

\begin{rem}
The generalised Catalan numbers 
\be
C^{\langle r\rangle }_k=\frac{r}{k+1}{rk+k \choose k}
\ee 
are related to the most known \emph{Fuss-Catalan} sequences 
\be
FC_{r,k}=\frac{1}{rk+1}{rk+k \choose k}.
\ee
  The Fuss-Catalan sequences also appear in random matrix theory and free probability as limiting moments of \emph{products} of independent Wishart matrices~\cite{Alexeev10,Liu11,Mlotkowski10}.
As shown by Penson and \.Zyczkowski~\cite{Penson11}, the Fuss-Catalan numbers $FC_{r,k}$ are the moment sequence of a distribution $F_{FC_r}(x)$ whose density is a Meijer G-function, with support on the interval $[0,L(r)]$.
\end{rem}
\par

\section{Proofs}
\label{sec:proof}
Theorem~\ref{thm:main} was discovered by a combination of experimenting and guessing, and the proof is based on matrix moments calculations combined with the exact solution for the enumeration of  $r$-plane trees (Proposition~\ref{prop:Gu}).  
\begin{proof}[{Proof of Theorem~\ref{thm:main}}] Let 
\be
m_{k,N}:=\int_{\R}x^k dF_N(x),\quad\text{and}\quad m_k
=\frac{1}{k+1}{(r+1)k \choose k}.
\ee
\begin{claim}
\label{cl:1}
Each $F_N$ has moments of all order.
\end{claim}
\begin{claim}
\label{cl:2}
For all $k$, the moment sequence $m_{k,N}$ converges to $m_k$, almost surely.
\end{claim}
\begin{claim}
\label{cl:3}
The sequence $(m_k)_{k\geq0}$ uniquely determine the distribution $F_{\langle r\rangle}$.
\end{claim}
By the  \emph{moment method}, the three claims imply Theorem~\ref{thm:main}.
Note that $F_N$ corresponds to a random measure with \emph{finite} support. Hence, Claim~\ref{cl:1} is immediate. Claim~\ref{cl:3} can be showed by checking Riesz's condition~\cite[Lemma B.2]{Bai2010}, 
\be
\liminf_k\frac{1}{k}m_{2k}^{\frac{1}{2k}}<\infty.
\label{eq:Riesz}
\ee
Indeed, Stirling's approximation formula implies
\be
\lim_{k\to\infty}\frac{1}{k}\left[\frac{1}{2k+1}{2(r+1)k \choose 2k}\right]^{\frac{1}{2k}}=0.
\ee
It remains to prove Claim~\ref{cl:2}:
\be
\lim_{N\to\infty}m_{k,N}=m_k,\quad \text{almost surely}.
\label{eq:cl2}
\ee
 Note that
\be
m_{k,N}=\int_{\R}x^kdF_N(x)=\frac{1}{rN}\Tr W_N^k.
\ee
We may first assume that the entries $X_{ij}$ are uniformly bounded.  Under this assumption 
we can prove the following two Lemmas.
\begin{lem}
\label{lem:exp} Let $\glN=N\staircase_r$. Then, for all $k\in\mathbb{N}$, 
\be
\lim_{N\to\infty}\E m_{k,N}=m_k,
\ee
\end{lem}
\begin{lem}
\label{lem:var} 
 Let $\glN=N\staircase_r$. Then, for all $k\in\mathbb{N}$, 
\be
\E\left(m_{k,N}-\E m_{k,N}\right)^2=\Or\left(\frac{1}{N^2}\right).
\ee
\end{lem}
Lemma~\ref{lem:exp} states the convergence in expectation of the moments. Lemma~\ref{lem:var} implies, through Borel-Cantelli, the almost sure convergence~\eref{eq:cl2}.

We now show that the Theorem holds true under the sole hypothesis of finite second moment $\E|X_{ij}|^2=1$. This is done by the standard procedure of truncation. Fix a constant $C$ and consider the symmetric matrix $\hat{X}_N$ whose elements satisfy, for $1\leq,i,j\leq rN$
\be
\hat{X}_N(i,j)=\frac{X_{ij}1_{|X_{ij}|<C}-\E X_{ij}1_{|X_{ij}|<C}}{\sqrt{\E\left|X_{ij}1_{|X_{ij}|<C}-\E X_{ij}1_{|X_{ij}|<C}\right|^2}}.
\ee
The matrix $\hat{X}_N$ is a truncated and standardised version of $X_N$. Denote by $\hat{F}_N$ the empirical spectral distribution of $\hat{X}_N$.  
\par

Recall that~\cite[Appendix C]{Anderson10} the space of probability measures on $\R$ equipped with the weak topology is metrizable with the \emph{L\'evy distance}, defined for any pair of distribution functions as:
\be
\label{eq:Levy}
d_{\text{L\'evy}}(F,G):=\inf\{\epsilon>0\colon F(x-\epsilon)-\epsilon\leq G\leq F(x+\epsilon)+\epsilon, \text{ for all } x\in\R\}.
\ee
\par

For the reader's convenience, we state explicitly a bound that can be extracted
from the book of Bai and Silverstein~\cite[Equations (3.1.3)-(3.1.4)-(3.1.5) on p. 48]{Bai2010}.
For large enough $C>0$,
\begin{multline}
\limsup_N\left( d_{\text{L\'evy}}(\hat{F}_N,{F}_N)\right)^4\leq 4\E\left(\left|X_{ij}\right|^2 1_{\left| X_{ij}\right|>C}\right)\nonumber\\+2\left(1- \left(\E\left|X_{ij}1_{|X_{ij}|<C}-\E X_{ij}1_{|X_{ij}|<C}\right|^2\right)^2\right),
\end{multline}
almost surely.
The above inequality is an application of classic matrix inequalities and the strong law of large numbers. 
Note that the right-hand side of the inequality can be made arbitrarily small, by choosing $C$ large, uniformly in $i,j$. 
From the fact that weak convergence is equivalent to convergence with respect to the L\'evy metric, we conclude that $\hat{F}_N$ and ${F}_N$ have the same limit in distribution.  (See also~\cite[Theorem 2.1.21]{Anderson10}.) 
\end{proof}

It remains to prove the combinatorial lemmas. We follow the scheme of proof of~\cite[Theorem 3.7]{Bai2010}, and its variant by Cheliotis~\cite{Cheliotis18} (arXiv version).
\begin{proof}[Proof of Lemmas~\ref{lem:exp} and~\ref{lem:var}]
We start from the exact formula
\begin{gather}
\E m_{k,N}= \E\int_{\R}x^kdF_N(x)
=\frac{1}{rN}\E\Tr W_N^k \nonumber\\
=\frac{1}{rN^{k+1}}\sum_{\substack{i(1),\ldots,i(k)\\ j(1),\ldots,j(k)}=1}^{rN} \E (X_N)_{i(1)j(1)}\overline{ (X_N)_{i(2)j(1)}} \cdots  (X_N)_{i(k)j(k)}\overline{ (X_N)_{i(1)j(k)}}
\label{eq:exact_formula_proof12}
\end{gather}
Now, we group row and column indices within the same block as follows
\begin{align}
\label{eq:block_map}
b\colon [rN]&\longrightarrow[r]\\
i&\longmapsto  b(i)=\left\lceil \frac{i}{N}\right\rceil \nonumber.
\end{align}
(We use the standard notation $[m]:=\{1,\ldots,m\}$.) 
We introduce the following subset of multiindices
\be
\Lambda:=\left\{ (i,j)\in [rN]^k\times[rN]^k \colon 
\begin{array}{@{}c@{}}b(i(m))+b(j(m))\leq r+1\\ b(i({m+1}))+b(j(m))\leq r+1
\end{array} \text{ for all } m\in[k] \right\}.
\ee
If $(i,j)\notin\Lambda$, then 
the corresponding word in the summation~\eqref{eq:exact_formula_proof12} contains a letter 
identically zero.
\par
For two $k$-tuples $i,j\colon[k]\to[rN]$, we associate the bipartite graph $G(i,j)$ with vertex set 
$$
V(i,j)=\{(1,i(1)),(1,i(2))\ldots, (1,i(k)),(2,j(1)),(2,j(2))\ldots, (2,j(k))\}$$
(its cardinality is not necessarily $2k$ because of possible repetitions), and set of 
 edges 
$$E(i,j)=\{\left(2m-1,\{(1,i(m)),(2,j(m))\}\right),\left(2m,\{(2,j(m)),(1,i(m+1))\}\right)\colon m\in[k]\},$$
with the convention $i(k+1)\equiv  i(1)$.  
Two graphs $G(i,j)$ and $G(i',j')$ are said isomorphic if one becomes the other by renaming the vertices, that is if
$
i=\sigma\circ i'$, and $j=\tau\circ j'
$
for some permutations $\sigma,\tau\in S_{rN}$.

With this notation, we write 
\begin{gather}
\label{eq:exact_formula_proof2}
\frac{1}{rN}\E\Tr W_N^k \nonumber
=\frac{1}{rN^{k+1}}\sum_{i,j\colon[k]\to[rN]} \E X_{G(i,j)}\,1_{(i,j)\in\Lambda},
\end{gather}
where we encode the word 
\be\label{eq:exact_formula_proof3}
X_{G(i,j)}:=X_{i(1)j(1)}\overline{X_{i(2)j(1)}}X_{i(2)j(2)}\overline{X_{i(3)j(2)}}\cdots X_{i(k)j(k)}\overline{X_{i(1)j(k)}}
\ee 
in the bipartite graph $G(i,j)$. From $G(i,j)$ we generate its \emph{skeleton}  $G_1(i,j)$  by identifying edges with equal ends. Formally, $G_1(i, j)$ has vertex set $V (i, j)$, and edge set 
$$\{\{(1,i(m)),(2,j(m))\},\{(2,j(m)),(1,i(m+1))\}\colon m\in[k]\}.$$
Note that if $G(i,j)$ contains an edge that does not appear at least twice, than $\E X_{G(i,j)}=0$ because the $X_{ij}$'s are independent and centred. Thus we have at most $k$ edges in the skeleton $G_1(i,j)$ and so at most $k+1$ vertices. The number of isomorphic graphs $G_1(i,j)$ with $v\leq k+1$ vertices is $ (rN)^v(1+o(1))$. Thus the contribution to the
expectation of these terms is $\leq \frac{(rN)^v}{rN^{k+1}}(1+o(1))$ which vanishes as $N\to\infty$, unless $v=k+1$. In the latter case $G_1(i,j)$ is a tree.
\par

Therefore, the indices $(i,j)$  that contribute in the large-$N$ limit of~\eqref{eq:exact_formula_proof12} are those for which
\begin{enumerate}
\item $(i,j)\in\Lambda$;
\item the associated $G_1(i,j)$ has exactly $k+1$ vertices;
\item the closed path  
$$
(1,i(1))\to (2,j(1)) \to (1,i(2))\to (2,j(2))\to\cdots\to (1,i(k))\to (2,j(k))\to (1,i(1))
$$ 
traverses each edge of the tree exactly twice.
\end{enumerate}
In fact, such a pair $(i,j)$ defines a {\em plane tree}, that is, a tree on which we have specified an order among the children of each vertex. (Among two vertices with common parent, we declare smaller the one that appears first in the sequence $(i(1),j(1),i(2),j(2),\ldots,i(k),j(k))$, that is, the one that is visited first in the path.) So the graph $G_1(i,j)$ can be identified with a  plane tree $T$ with $k+1$ vertices.  
Since the  $X_{ij}$'s are independent and have unit variance, the corresponding word has expectation $\E X_{G(i,j)}=1$, if $(i,j)\in\Lambda$. 
Each label $m\in [rN]$ can be written as $m=(b-1)N+p$ for a unique choice of `block index' $b\in[r]$ and `inner index' $p\in[N]$. Therefore, for any  choice of block indices, there are $N^{k+1}[1+o(1)]$  ways to choose multiindices $i$ and $j$ corresponding to the same plane tree $T$ up to isomorphism. The condition $1_{(i,j)\in\Lambda}$ is the colouring condition on $T$ that makes it  a $r$-plane tree:
 \begin{gather}
\label{eq:exact_formula_proof4}
\frac{1}{rN}\E\Tr W_N^k \nonumber
=\frac{1}{rN^{k+1}}\sum_{\substack{\text{plane trees $T$}\\ \text{on $k+1$ vertices}} }\sum_{\substack{c\colon [k+1]\to [r]\\c(u)+c(v)\leq r+1\\ \text{for all edges  $(u,v)\in T$}}} N^{k+1}[1+o(1)] \\
=\frac{1}{r}\#\{\text{$r$-plane trees  on $k+1$ vertices}\}[1+o(1)]. 
\end{gather}
We now recall Proposition~\ref{prop:Gu}
\be
\#\{\text{$r$-plane trees  on $k+1$ vertices}\}=\frac{r}{k+1}{rk+k\choose k},
\ee
 and conclude the proof of Lemma~\ref{lem:exp}.
 \par
 
To prove Lemma~\ref{lem:var}, we write
\begin{gather}
\frac{1}{(rN)^2}\left[\E\left( \Tr W_N^k\right)^2-\left(\E \Tr W_N^k\right)^2\right]\nonumber\\
=
\frac{1}{(rN)^2}\frac{1}{N^{2k}}\sum_{i,j,i',j'\colon[k]\to[rN]} \left[\E X_{G(i,j)}X_{G(i',j')}-\E X_{G(i,j)}\E X_{G(i',j')}\right]\,1_{(i,j)\in\Lambda}1_{(i',j')\in\Lambda}\nonumber\\
\leq
\frac{1}{(rN)^2}\frac{1}{N^{2k}}\sum_{i,j,i',j'\colon[k]\to[rN]} \left[\E X_{G(i,j)}X_{G(i',j')}-\E X_{G(i,j)}\E X_{G(i',j')}\right].
\end{gather}
The latter can be shown to be $\Or((rN)^{-2})$, as detailed e.g. in~\cite[p. 50]{Bai2010}.
\end{proof}
\par

\begin{proof}[Proof of Proposition~\ref{prop:G_Beta}]
To prove i) we assume that $G_{\langle r\rangle}(z)$ can be expanded in a neighbourhood of $z=\infty$
\be
G_{\langle r\rangle}(z)=\int_{\R}\frac{1}{z-x}\rmd F_{\langle r\rangle}(x)=\sum_{k=0}^{\infty}\frac{1}{z^{k+1}}\int_{\R} x^k\rmd F_{\langle r\rangle}(x)
=\sum_{k=0}^{\infty}m_k\frac{1}{z^{k+1}}.
\ee
(It can be verified, a posteriori, that this expansion holds true for $|z|>L(r)$.) When computing the ratio of consecutive terms of the series we identify the claimed hypergeometric representation~\eqref{eq:Hypergeometric}.
\par

In order to prove ii), recall the moments of Beta and uniform random variables,
\be
\label{eq:mom_beta_unif}
\E\Beta(a,b)^k=\frac{\Gamma(a+b)\Gamma(a+k)}{\Gamma(a)\Gamma(a+b+k)}=\prod_{i=0}^{k-1}\frac{a+i}{a+b+i},\qquad 
\E\Unif(0,\ell)^k=\frac{\ell^{k}}{k+1}.
\ee
A little calculation shows that
\begin{gather}
m_k=\frac{[(r+1)k]!}{(rk)!(k+1)!}=\displaystyle\frac{\displaystyle\prod_{i=0}^{(r+1)k-1}\left((r+1)k-i\right)}{\displaystyle\prod_{i=0}^{rk-1}\left(rk-i\right)
\prod_{i=0}^{k}\left(k+1-i\right)}\nonumber \\
=\left(L(r)\right)^k
\displaystyle\frac{\displaystyle\prod_{i=0}^{(r+1)k-1}\left(k-\frac{i}{r+1}\right)}{\displaystyle\prod_{i=0}^{rk-1}\left(k-\frac{i}{r}\right)
\displaystyle\prod_{i=0}^{k}\left(k+1-i\right)}
=\frac{\left(L(r)\right)^k}{k+1}
\displaystyle\prod_{j=1}^{r}\prod_{i=0}^{k-1}\frac{\left(\frac{j}{r+1}+i\right)}{\left(\frac{j}{r}+i\right)}.
\end{gather}
Compare now with~\eref{eq:mom_beta_unif} to conclude the proof. 
\par

Finally, for point iii) we notice that
\begin{gather}
m_k=\frac{1}{k+1}{(r+1)k \choose k}
=\frac{1}{k+1}\times{\rm coefficient\, of}\, z^k\, {\rm in}\left(1+z\right)^{(r+1)k}\nonumber\\
=\frac{1}{k+1}\cdot\frac{1}{2\pi \rmi}\oint_{|z|=1}\frac{\left(1+z\right)^{(r+1)k}}{z^{k+1}}\rmd z.
\end{gather}
Now change coordinates $z\mapsto \exp({2\pi\rmi u})$ to get
\be
m_k=\int_{0}^1u^kdu\cdot\int_{0}^1\left[\exp(-2\pi\rmi u') \left(1+\exp(2\pi\rmi u') \right)^{r+1}\right]^kdu'.
\ee
\end{proof}

\begin{proof}[Proof of Corollary~\ref{cor:supp}]
By Proposition~\ref{prop:G_Beta}, the density $F'_{\langle r\rangle}$ is the Mellin (multiplicative) convolution of the densities of 
$\Unif\left(0,L(r)\right),\Beta\left(\frac{1}{r+1},\frac{1}{r(r+1)}\right),\ldots, \Beta\left(\frac{r}{r+1},\frac{r}{r(r+1)}\right)$:
\be
\label{eq:proof_cor1}
F'_{\langle r\rangle}(x)=C\int_{[0,1]^r} \chi_{[0,L(r)]}\left(\frac{x}{x_1}\right)
\prod_{j=1}^r\left(\frac{x_j}{x_{j+1}}\right)^{\frac{j}{r+1}-1}\left(1-\frac{x_j}{x_{j+1}}\right)^{\frac{j}{r(r+1)}-1}\frac{dx_j}{x_j}, 
\ee
where
\be
C= \frac{1}{L(r)}\prod_{j=1}^r\displaystyle\frac{\Gamma\left(\frac{j}{r}\right)}{\Gamma\left(\frac{j}{r+1}\right)\Gamma\left(\frac{j}{r(r+1)}\right)},
\ee
 $\chi_{[0,L(r)]}$ is the characteristic function of the interval $[0,L(r)]$, and $x_{r+1}\equiv1$. Since  the beta random variables are supported on $[0,1]$, we see immediately that $F'_{\langle r\rangle}(x)$ is zero for $x\notin [0,L(r)]$.  Assume that $0< x< L(r)$. Then, we can rewrite~\eqref{eq:proof_cor2} as
 \be
 \label{eq:proof_cor2}
F'_{\langle r\rangle}(x)=C\int_{x/L(r)}^1\frac{dx_1}{x_1}\int_{x_2}^1\frac{dx_2}{x_2}\cdots \int_{x_{r-1}}^1\frac{dx_r}{x_r}
\prod_{j=1}^r\left(\frac{x_j}{x_{j+1}}\right)^{\frac{j}{r+1}-1}\left(1-\frac{x_j}{x_{j+1}}\right)^{\frac{j}{r(r+1)}-1}.
\ee
For $x\searrow 0$, the leading contribution to the integral comes from the singular part of the integrand at $x=0$. Therefore, for $x\searrow0$ we have
  \be
 \label{eq:proof_cor3}
F'_{\langle r\rangle}(x)\sim C'\int_{x/L(r)}^1\frac{dx_1}{x_1}\int_{x_2}^1\frac{dx_2}{x_2}\cdots \int_{x_{r-1}}^1\frac{dx_r}{x_r}
\prod_{j=1}^r\left(\frac{x_j}{x_{j+1}}\right)^{\frac{j}{r+1}-1}= c x^{-\frac{r}{r+1}}.
\ee
For $x\nearrow L(r)$, the leading contribution to the integral comes instead from the vicinity of $x=L(r)$, picking the singular part of the integrand for $x_1=\cdots=x_r=1$. Therefore, for $x\nearrow L(r)$ we have
  \be
 \label{eq:proof_cor4}
F'_{\langle r\rangle}(x)\sim C''
\int_{x/L(r)}^1dx_1\int_{x_2}^1dx_2\cdots \int_{x_{r-1}}^1dx_r
\prod_{j=1}^r\left(x_{j+1}-x_j\right)^{\frac{j}{r(r+1)}-1}= c' \sqrt{L(r)-x}.
\ee
\end{proof}
\par

\section{Outlook}

We conclude with some exercises and further food for thought.

\begin{enumerate}
\item Show that adding/removing a finite number of boxes independent of $N$ to the dilation of the staircase partition does not change the limiting spectral distribution $F_{\langle r\rangle}$. 
\item  Prove Theorem~\ref{thm:main} using the Stieltjes transform method. One should first identify the algebraic equation whose solution (decaying at infinity as $1/z$) is $G_{\langle r\rangle}(z)$.
\item  Investigate the microscopic statistics of the eigenvalues of block-shaped random matrices.
\item For $\lambda$-shaped random matrices one expects a relation between the limit shape $\lambda$ and the limiting spectral distribution. Find new examples, for which the limit of $(F_N)_{N\geq1}$ can be computed explicitly.
\end{enumerate}

\ack
The authors acknowledge the partial support by the Italian National Group of Mathematical Physics INdAM-GNFM. FDC acknowledges the support by Regione Puglia through the project `Research for Innovation' - UNIBA024, and by INFN through the project `QUANTUM'.  ML acknowledges the support by PNRR MUR project CN00000013-`Italian National Centre on HPC, Big Data and Quantum Computing'. 
 FDC thanks Neil O'Connell for fruitful discussions at an early stage of this work.


\section*{References} 
 \begin{thebibliography}{99}
 
 \bibitem{Adler13} 
 Adler M, Van Moerbeke P, Wang D,
 Random matrix minor processes related to percolation theory,
 \emph{Random Matrices: Theory and Applications}
{\bf 2}, 1350008 (2013).

\bibitem{Alexeev10}
Alexeev N, G\"otze F and Tikhomirov A, 
Asymptotic distribution of singular values of powers of random matrices, 
\emph{Lithuanian Math. J.} {\bf 50}, 121 (2010). 

 
 \bibitem{Anderson10}
Anderson G W,  Guionnet A and Zeitouni O,
\emph{An Introduction to Random Matrices},
Cambridge University Press, 2010.
 
 \bibitem{Bai2010}
  Bai Z, Silverstein J W,
Spectral Analysis of Large Dimensional Random Matrices, 2dn Edition,  Springer, 2010.
 
\bibitem{Ballico15} Ballico E, 
Linear subspaces of matrices associated to a Ferrers diagram and with a prescribed lower bound for their rank, \emph{Linear Algebra and its Applications} {\bf 483}, 30-39 (2015).

\bibitem{Cheliotis18}
Cheliotis D,
Triangular random matrices and biorthogonal ensembles,
\emph{Statistics and Probability Letters} {\bf 134}, 36-44  (2018); 
arXiv:1404.4730v1.

\bibitem{Collins14}
Collins B, Gawron P,  Litvak A E, \.Zyczkowski K,
Numerical range for random matrices,
\emph{J. Math. Anal. Appl.} {\bf 418}, 516-533 (2014).


\bibitem{Cunden19}
Cunden F D, Majumdar S N, O'Connell N,
Free fermions and $\alpha$-determinantal processes,
\emph{J. Phys. A: Math. Theor.} {\bf 52}, 165202 (2019).


\bibitem{Dolega}  Dolega M,  F\'eray V,  \'Sniady P, 
Explicit combinatorial interpretation of Kerov character polynomials as numbers of permutation factorizations, 
\emph{Adv. Math.} {\bf 225}, 81-120 (2010).

\bibitem{Dunkl13}
Dunkl C F,
\emph{Products of Beta distributed random variables},
arXiv:1304.6671.

\bibitem{Dykema04} Dykema K and Haagerup U,
 DT-operators and decomposability of Voiculescu's circular operator, 
 \emph{Amer. J. Math.} {\bf126(1)},121-189 (2004).
 
 
 \bibitem{Feray11} F\'eray V,  \'Sniady P, 
 Asymptotics of characters of symmetric groups related to Stanley character formula,
 \emph{Ann. Math.} {\bf  173}, 887-906  (2011).
 
 \bibitem{Flynn-Connolly18} Flynn-Connolly O,
 Random matrices, genus expansions and the symmetric group, \emph{UCD Summer Research Project 2018 - Final report}, 2018.
 
 \bibitem{Forrester10}
Forrester P J, Log-gases and random matrices, 
London Mathematical Society Monographs Series, {\bf 34}, Princeton University Press, Princeton,  2010.
 
 \bibitem{Forrester17}
Forrester P J,  Wang D,
 Muttalib-Borodin ensembles in random matrix theory - realisations and correlation functions,
 \emph{Electron. J. Probab.} {\bf 22 }, 1-43 (2017).
 
 \bibitem{Gu09} Gu N S S, Prodinger H, 
 Bijections for $2$-plane trees and ternary trees, 
 \emph{European J. Combin.} {\bf 30 (4)},  969-985 (2009).
 
 \bibitem{Gu10} Gu N S S, Prodinger H,  Wagner S, 
 Bijections for a class of labeled plane trees,
 \emph{European J. Combin.} {\bf  31}, 720-732  (2010).
 
 \bibitem{Ledoux04}
Ledoux M, 
Differential operators and spectral distributions of invariant ensembles from the
classical orthogonal polynomials. The continuous case, 
\emph{Elec. J. Probab.} {\bf 9}, 177-208 (2004).

\bibitem{Liu11}
Liu D-Z, Song C, Wang Z-D, 
On explicit probability densities associated with
Fuss-Catalan numbers, 
\emph{Proc. AMS} {\bf 39}, (2011).

\bibitem{Livan18} Livan G, Novaes M, Vivo P,
Introduction to Random Matrices - 
Theory and Practice, SpringerBriefs in Mathematical Physics (2018).


\bibitem{Macdonald} Macdonald I G, Symmetric Functions and Hall Polynomials, 2nd Ed., Oxford Mathematical Monographs, 1995.

\bibitem{Marchenko67} 
Marchenko V A and Pastur L A,
 \emph{Distribution of eigenvalues for some sets of random matrices}, 
 Matematicheskii Sbornik, {\bf 114(4)}, 507-536(1967).

\bibitem{Mlotkowski10}
M\l otkowski W,
Fuss-Catalan numbers in noncommutative probability,
\emph{Documenta Mathematica} {\bf 15}, 939-955, (2010).

\bibitem{Mlotkowski13} 
M\l otkowski W, Penson K A,
The probability measure corresponding to 2-plane trees,
\emph{Probability and Mathematical Statistics} {\bf 33 (2)}, 255-264 (2013).

\bibitem{Mlotkowski14} 
M\l otkowski W, Penson K A, 
Probability distributions with binomial moments,
\emph{Inf. Dim. Analysis, Quantum Prob. and Related Topics} {\bf 17/2}  (2014).

\bibitem{Muirhead09}
Muirhead R J, 
Aspects of multivariate statistical theory, 
John Wiley \& Sons (2009).

\bibitem{Nakashima20}
Nakashima H, Graczyk P, 
Wigner and Wishart Ensembles for graphical models,
\emph{arXiv:2008.10446}.

 \bibitem{OEIS}  
 OEIS Foundation Inc. (2023), 
\emph{The On-Line Encyclopedia of Integer Sequences}, Published electronically at http://oeis.org.

\bibitem{Okoth22} Okoth I O,  Wagner S, 
Refined enumeration of $k$-plane trees and $k$-noncrossing trees, 
\emph{arXiv:2205.01002}.

\bibitem{Penson11} Penson K A  and  \.Zyczkowski K,
Product of Ginibre matrices: Fuss-Catalan and Raney distributions,
\emph{Phys. Rev. E} {\bf 83}, 061118 (2011).


\bibitem{Stanley04} Stanley R P,
Irreducible Symmetric Group Characters
of Rectangular Shape
\emph{S\'eminaire Lotharingien de Combinatoire}
{\bf 20}, Article B50d  (2004).

\bibitem{Stanley15} 
R. Stanley R P, 
\emph{Catalan Numbers}, Cambridge
University Press, 2015. 


 \end{thebibliography}
\end{document}